\documentclass[12pt, a4]{amsart}%
\usepackage{amsfonts}
\usepackage{amssymb}
\usepackage{amsmath}
\usepackage{amsthm}
\usepackage{fullpage}
\providecommand{\U}[1]{\protect\rule{.1in}{.1in}}
\theoremstyle{plain}
\newtheorem{theorem}{Theorem}[section]
\newtheorem{lemma}[theorem]{Lemma}
\theoremstyle{definition}
\newtheorem{definition}[theorem]{Definition}
\newtheorem{example}[theorem]{Example}
\newtheorem{remark}[theorem]{Remark}

\numberwithin{equation}{section}

\begin{document}
\title[Systems with nonlinear coupled nonlocal initial conditions]{Existence results for systems with nonlinear coupled nonlocal initial conditions}
\author[O. Bolojan-Nica]{Octavia Bolojan-Nica}
\address{Octavia Bolojan-Nica, Departamentul de Matematic\u{a}, Universitatea
Babe\c{s}-Bolyai, Cluj 400084, Romania}
\email{octavia.nica@math.ubbcluj.ro}
\author[G. Infante]{Gennaro Infante}
\address{Gennaro Infante, Dipartimento di Matematica ed Informatica, Universit\`{a}
della Calabria, 87036 Arcavacata di Rende, Cosenza, Italy}
\email{gennaro.infante@unical.it}
\author[R. Precup]{Radu Precup}
\address{Radu Precup, Departamentul de Matematic\u{a}, Universitatea Babe\c{s}-Bolyai,
Cluj 400084, Romania}
\email{r.precup@math.ubbcluj.ro}
\date{}
\maketitle

\begin{abstract}
The purpose of the present work is to study the existence of solutions to
initial value problems for nonlinear first order differential systems with
nonlinear nonlocal boundary conditions of functional type. The existence
results are established by means of the Perov, Schauder and Leray-Schauder
fixed point principles combined with a technique based on vector-valued
metrics and convergent to zero matrices.

\end{abstract}

\section{Introduction}

Nonlocal problems for different classes of differential equations and systems
are intensively studied in the literature by a variety of methods (see for
example \cite{aizicovici-lee, avalshv1, avalshv2, ref4, boucherif2,
boucherif-ejde, ref6, ref7, ref8, jankowski, prec-nica, nt, ref11, webb-lan,
webb, webb-infante, webb-infante1, webb-infante2, xue1, xue2} and references
therein). For problems with nonlinear boundary conditions we refer the reader
to~\cite{amp, Cabada1, Cab-Ter, acfm-nonlin, dfdorjp, Goodrich3, Goodrich4,
Goodrich5, gi-caa, gipp-cant, gipp-nonlin, gipp-mmas, paola} and references therein.

In this paper we extend the results from \cite{bip, nica-EJDE, nica-func-w},
in order to deal with the nonlocal initial value problem for the first order
differential system
\begin{equation}
\left\{
\begin{array}
[c]{l}%
x^{\prime}\left(  t\right)  =f_{1}\left(  t,x\left(  t\right)  ,y(t)\right)
,\ \ \\
y^{\prime}\left(  t\right)  =f_{2}\left(  t,x\left(  t\right)  ,y(t)\right)
,\ \ \ \ \ \text{ on }\left[  0,1\right]  ,\\
x\left(  0\right)  =\alpha\lbrack x,y],\\
y\left(  0\right)  =\beta\lbrack x,y].
\end{array}
\right.  \label{1}%
\end{equation}
Here, \ $f_{1},f_{2}:[0,1]\times\mathbf{%
%TCIMACRO{\U{211d} }%
%BeginExpansion
\mathbb{R}
%EndExpansion
}^{2}\rightarrow\mathbf{%
%TCIMACRO{\U{211d} }%
%BeginExpansion
\mathbb{R}
%EndExpansion
}$ are continuous functions and $\alpha,\beta:(C[0,1])^{2}\rightarrow
%\mathbf{%
%TCIMACRO{\U{211d} }%
%BeginExpansion
\mathbb{R}
%EndExpansion
%
%,~}
$ are nonlinear continuous functionals.

\begin{remark}
We can also consider the case when $f_{1},f_{2}$ are $L^{1}$-Carath\'{e}odory functions
and $\alpha,\beta:\left(  W^{1,1}(0,1)\right)  ^{2}\rightarrow
%\mathbf{%
%TCIMACRO{\U{211d} }%
%BeginExpansion
\mathbb{R}
%EndExpansion
%
%,~}i=1,2
$. In this case, the system (\ref{1}) can be defined a.e. on $\left[  0,1\right]
$ and the solutions can be sought in the Sobolev space $\left(  W^{1,1}%
(0,1)\right)  ^{2}.$
\end{remark}

Our approach is to rewrite the problem (\ref{1}) as a system of the form
\begin{align*}
x_{a}  &  =\left(  a+\int_{0}^{t}f_{1}\left(  s,x\left(  s\right)  ,y\left(
s\right)  \right)  ds,\ \alpha\left[  x,y\right]  \right)  ,\\
y_{b}  &  =\left(  b+\int_{0}^{t}f_{2}\left(  s,x\left(  s\right)  ,y\left(
s\right)  \right)  ds,\ \beta\left[  x,y\right]  \right)  ,
\end{align*}
where by $x_{a},y_{b}$ we mean the pairs $\left(  x,a\right)  ,\left(
y,b\right)  \in C\left[  0,1\right]  \times%
%TCIMACRO{\U{211d} }%
%BeginExpansion
\mathbb{R}
%EndExpansion
.$

This, in turn, can be viewed as a fixed point problem in $\left(  C\left[
0,1\right]  \times%
%TCIMACRO{\U{211d} }%
%BeginExpansion
\mathbb{R}
%EndExpansion
\right)  ^{2}$ for the completely continuous operator
\[
T=(T_{1},T_{2}):\left(  C\left[  0,1\right]  \times%
%TCIMACRO{\U{211d} }%
%BeginExpansion
\mathbb{R}
%EndExpansion
\right)  ^{2}\rightarrow\left(  C\left[  0,1\right]  \times%
%TCIMACRO{\U{211d} }%
%BeginExpansion
\mathbb{R}
%EndExpansion
\right)  ^{2},
\]
where $T_{1}$ and $T_{2}$\ are given by%
\[%
\begin{array}
[c]{l}%
T_{1}\left[  x_{a},y_{b}\right]  =\left(  a+\int_{0}^{t}f_{1}\left(
s,x\left(  s\right)  ,y(s)\right)  ds,\ \alpha\lbrack x,y]\right)  ,\\
T_{2}\left[  x_{a},y_{b}\right]  =\left(  b+\int_{0}^{t}f_{2}\left(
s,x\left(  s\right)  ,y(s)\right)  ds,\ \beta\lbrack x,y]\right)  .
\end{array}
\]

In what follows, we introduce some notations, definitions and basic results
which are used throughout this paper. Three different fixed point principles
are used in order to prove the existence of solutions for the problem
\eqref{1}, namely the fixed point principles of Perov, Schauder and
Leray-Schauder (see \cite{ref11, p1}). A technique that makes use of the
vector-valued metrics and convergent to zero matrices has an essential role in
all three cases. Therefore, we recall the fundamental results that are used in
the next sections (see \cite{ref12, op, ref11}).\medskip

Let $X$ be a nonempty set.

\begin{definition}
By a \emph{vector-valued metric} on $X$ we mean a mapping $d:X\times
X\rightarrow\mathbf{%
%TCIMACRO{\U{211d} }%
%BeginExpansion
\mathbb{R}
%EndExpansion
}_{+}^{n}$ such that\medskip

(i) $d(u,v)\geq0$ for all $u,v\in X$ and if $d(u,v)=0$ then $u=v;$

(ii) $d(u,v)=d(v,u)$ for all $u,v\in X;$

(iii) $d(u,v)\leq d(u,w)+d(w,v)$ for all $u,v,w\in X.\medskip$
\end{definition}

\noindent Here, if $x,y\in\mathbf{%
%TCIMACRO{\U{211d} }%
%BeginExpansion
\mathbb{R}
%EndExpansion
}^{n},$ $x=(x_{1},x_{2},...,x_{n}),$ $y=(y_{1},y_{2},...,y_{n}),$ by $x\leq y
$ we mean $x_{i}\leq y_{i}$ for $i=1,2,...,n.$ We call the pair $(X,d)$ a
\emph{generalised metric space}. For such a space convergence and completeness
are similar to those in usual metric spaces.\medskip

\begin{definition}
A square matrix $M$ with nonnegative elements is said to be \emph{convergent
to zero} if%
\[
M^{k}\rightarrow0\;\;\;\text{as\thinspace\ \ }k\rightarrow\infty.
\]

\end{definition}

The property of being convergent to zero is equivalent to each of the
following conditions from the characterisation lemma below (see \cite[pp 9,
10]{bp}, \cite{ref11}, \cite{p1}, \cite[pp 12, 88]{v}):

\begin{lemma}
Let $M$ be a square matrix of nonnegative numbers. The following statements
are equivalent:

(i) $M$ is a matrix convergent to zero;

(ii) $I-M$ is nonsingular and $(I-M)^{-1}=I+M+M^{2}+...$ (where $I$ stands for
the unit matrix of the same order as $M$);

(iii) the eigenvalues of $M$ are located inside the unit disc of the complex plane;

(iv) $I-M$ is nonsingular and $(I-M)^{-1}$ has nonnegative elements.\newline
\end{lemma}

Note that, according to the equivalence of the statements (i) and (iv), a
matrix $M$ is convergent to zero if and only if the matrix $I-M$ is
\textit{inverse-positive}.

The following lemma is a consequence of the previous characterisations.

\begin{lemma}
Let $A$ be a matrix that is convergent to zero. Then for each matrix $B$ of
the same order whose elements are nonnegative and sufficiently small, the
matrix $A+B$ is also convergent to zero.
\end{lemma}

\begin{definition}
Let $(X,d)$ be a generalized metric space. An operator $T:X\rightarrow X$ is
said to be \emph{contractive} (with respect to the vector-valued metric $d$ on
$X$) if there exists a convergent to zero (Lipschitz) matrix $M$ such that
\[
d(T(u),T(v))\leq Md(u,v)\;\,\,\mbox{ for all }u,v\in X.
\]

\end{definition}

\begin{theorem}
[Perov]Let $(X,d)$ be a complete generalized metric space and $T:X\rightarrow
X$ a contractive operator with Lipschitz matrix $M.$ Then $T$ has a unique
fixed point $u^{\ast}$ and for each $u_{0}\in X$ we have
\[
d(T^{k}(u_{0}),u^{\ast})\leq M^{k}(I-M)^{-1}d(u_{0},T(u_{0}))\,\mbox{ for
all }k\in\mathbb{N}.
\]

\end{theorem}

\begin{theorem}
[Schauder]Let $X$ be a Banach space, $D\subset X$ a nonempty closed bounded
convex set and $T:D\rightarrow D$ a completely continuous operator (i.e., $T$
is continuous and $T(D)$ is relatively compact). Then $T$ has at least one
fixed point.
\end{theorem}

\begin{theorem}
[Leray--Schauder]Let $(X,\left\vert .\right\vert _{X})$ be a Banach space,
$R>0$ and $T:\overline{B}_{X}(0;R)\rightarrow X$ a completely continuous
operator. If $\left\vert u\right\vert _{X}<R$ for every solution $u$ of the
equation $u=\lambda T(u)$ and any $\lambda\in(0,1),$ then $T$ has at least one
fixed point.
\end{theorem}

In this paper, by $\left\vert x\right\vert _{C}$, where $x\in C[0,1],$ we
mean
\[
\left\vert x\right\vert _{C}=\underset{t\in\lbrack0,1]}{\max}\left\vert
x(t)\right\vert .
\]
Also, throughout the paper we shall assume that $f_{1},f_{2}:[0,1]\times
\mathbf{%
%TCIMACRO{\U{211d} }%
%BeginExpansion
\mathbb{R}
%EndExpansion
}^{2}\rightarrow\mathbf{%
%TCIMACRO{\U{211d} }%
%BeginExpansion
\mathbb{R}
%EndExpansion
}$ are such that $\ f_{1}(.,x,y),f_{2}(.,x,y)$ are measurable for each
$\left(  x,y\right)  \in$\ $\mathbf{%
%TCIMACRO{\U{211d} }%
%BeginExpansion
\mathbb{R}
%EndExpansion
}^{2}$ and $f_{1}(t,.,.),f_{2}(t,.,.)$ are continuous for almost all
$t\in\lbrack0,1].$

\section{Existence and Uniqueness of the Solution}

\setcounter{equation}{0} In this section we show that the existence of
solutions to the problem (\ref{1}) follows from Perov's fixed point theorem in
case that the nonlinearities $f_{1},f_{2}$ and the functionals $\alpha,\beta$
satisfy Lipschitz conditions of the type$:$%

\begin{equation}
\left\{
\begin{array}
[c]{l}%
\left\vert f_{1}(t,x,y)-f_{1}(t,\overline{x},\overline{y})\right\vert \leq
a_{1}\left\vert x-\overline{x}\right\vert +b_{1}\left\vert y-\overline
{y}\right\vert \\
\left\vert f_{2}(t,x,y)-f_{2}(t,\overline{x},\overline{y})\right\vert \leq
a_{2}\left\vert x-\overline{x}\right\vert +b_{2}\left\vert y-\overline
{y}\right\vert ,
\end{array}
\right.  \label{200}%
\end{equation}
for all $x,y,\overline{x},\overline{y}\in%
%TCIMACRO{\U{211d} }%
%BeginExpansion
\mathbb{R}
%EndExpansion
,$ and%
\begin{equation}
\left\{
\begin{array}
[c]{c}%
\left\vert \alpha\left[  x,y\right]  -\alpha\left[  \overline{x},\overline
{y}\right]  \right\vert \leq A_{1}\left\vert x-\overline{x}\right\vert
_{C}+B_{1}\left\vert y-\overline{y}\right\vert _{C}\\
\left\vert \beta\left[  x,y\right]  -\beta\left[  \overline{x},\overline
{y}\right]  \right\vert \leq A_{2}\left\vert x-\overline{x}\right\vert
_{C}+B_{2}\left\vert y-\overline{y}\right\vert _{C},
\end{array}
\right.  \label{201}%
\end{equation}
for all $x,y,\overline{x},\overline{y}\in$\ $C[0,1].\medskip$

For a given number $\theta>0,$ denote%
\[%
\begin{array}
[c]{ll}%
m_{11}\left(  \theta\right)  =\max\left\{  \frac{1}{\theta},a_{1}+\theta
A_{1}\right\}  & m_{12}\left(  \theta\right)  =b_{1}+\theta B_{1}\\
m_{21}\left(  \theta\right)  =a_{2}+\theta A_{2} & m_{22}\left(
\theta\right)  =\max\left\{  \frac{1}{\theta},b_{2}+\theta B_{2}\right\}  .
\end{array}
\]

\begin{theorem}
Assume that $f_{1},f_{2}$ satisfy the Lipschitz conditions \emph{(\ref{200})}
and $\alpha,\beta$ satisfy conditions\emph{\ (\ref{201}). }In addition assume
that for some $\theta>0,$ the matrix
\begin{equation}
M_{\theta}=\left[
\begin{array}
[c]{cc}%
m_{11}\left(  \theta\right)  & m_{12}\left(  \theta\right) \\
m_{21}\left(  \theta\right)  & m_{22}\left(  \theta\right)
\end{array}
\right]  \label{m}%
\end{equation}
is convergent to zero. Then the problem \emph{(\ref{1})} has a unique solution.
\end{theorem}

\begin{proof}
We shall apply Perov's fixed point theorem in $\left(  C\left[  0,1\right]
\times%
%TCIMACRO{\U{211d} }%
%BeginExpansion
\mathbb{R}
%EndExpansion
\right)  ^{2}$ endowed with the vector-valued norm $\left\Vert .\right\Vert
_{\left(  C\left[  0,1\right]  \times%
%TCIMACRO{\U{211d} }%
%BeginExpansion
\mathbb{R}
%EndExpansion
\right)  ^{2}},$
\[
\left\Vert u\right\Vert _{\left(  C\left[  0,1\right]  \times%
%TCIMACRO{\U{211d} }%
%BeginExpansion
\mathbb{R}
%EndExpansion
\right)  ^{2}}=\left[
\begin{array}
[c]{c}%
\left\vert x_{a}\right\vert \\
\left\vert y_{b}\right\vert
\end{array}
\right]  ,
\]
for $u=\left(  x_{a},y_{b}\right)  .$ Here%
\[
\left\vert x_{a}\right\vert =\left\vert (x,a)\right\vert =\left\vert
x\right\vert _{C}+\theta\left\vert a\right\vert,
\]
which represents a norm on $C\left[  0,1\right]  \times%
%TCIMACRO{\U{211d} }%
%BeginExpansion
\mathbb{R}
%EndExpansion
.$ We have to prove that $T$ is contractive with respect to the convergent to
zero matrix $M_{\theta},$ more exactly that%
\[
\left\Vert T(u)-T(\overline{u})\right\Vert _{\left(  C\left[  0,1\right]
\times%
%TCIMACRO{\U{211d} }%
%BeginExpansion
\mathbb{R}
%EndExpansion
\right)  ^{2}}\leq M_{\theta}\left\Vert u-\overline{u}\right\Vert _{\left(
C\left[  0,1\right]  \times%
%TCIMACRO{\U{211d} }%
%BeginExpansion
\mathbb{R}
%EndExpansion
\right)  ^{2}},
\]
for all $u=(x_{a},y_{b}),\overline{u}=(\overline{x}_{\overline{a}}%
,\overline{y}_{\overline{b}})\in\left(  C\left[  0,1\right]  \times%
%TCIMACRO{\U{211d} }%
%BeginExpansion
\mathbb{R}
%EndExpansion
\right)  ^{2}.$ Indeed, we have%
\begin{align}
&  \left\vert T_{1}\left[  x_{a},y_{b}\right]  -T_{1}\left[  \overline
{x}_{\overline{a}},\overline{y}_{\overline{b}}\right]  \right\vert \nonumber\\
&  \leq\left\vert \int_{0}^{t}\left\vert f_{1}\left(  s,x\left(  s\right)
,y(s)\right)  -f_{1}\left(  s,\overline{x}\left(  s\right)  ,\overline
{y}(s)\right)  \right\vert ds\right\vert _{C}+\left\vert a-\overline
{a}\right\vert +\theta\left\vert \alpha\lbrack x,y]-\alpha\lbrack\overline
{x},\overline{y}]\right\vert \nonumber\\
&  \leq\left\vert a_{1}\int_{0}^{t}\left\vert x(s)-\overline{x}(s)\right\vert
ds+b_{1}\int_{0}^{t}\left\vert y(s)-\overline{y}(s)\right\vert \right\vert
_{C}+\theta A_{1}\left\vert x-\overline{x}\right\vert _{C}+\theta
B_{1}\left\vert y-\overline{y}\right\vert _{C}+\left\vert a-\overline
{a}\right\vert \nonumber\\
&  \leq\left(  a_{1}+\theta A_{1}\right)  \left\vert x-\overline{x}\right\vert
_{C}+\left(  b_{1}+\theta B_{1}\right)  \left\vert y-\overline{y}\right\vert
_{C}+\frac{1}{\theta}\cdot\theta\left\vert a-\overline{a}\right\vert
\nonumber\\
&  \leq\max\left\{  \frac{1}{\theta},a_{1}+\theta A_{1}\right\}  \left\vert
x_{a}-\overline{x}_{\overline{a}}\right\vert +\left(  b_{1}+\theta
B_{1}\right)  \left\vert y_{b}-\overline{y}_{\overline{b}}\right\vert
\nonumber\\
&  =m_{11}\left(  \theta\right)  \left\vert x_{a}-\overline{x}_{\overline{a}%
}\right\vert +m_{12}\left(  \theta\right)  \left\vert y_{b}-\overline
{y}_{\overline{b}}\right\vert . \label{T1_eval}%
\end{align}
Similarly, we have%
\begin{align}
&  \left\vert T_{2}\left[  x_{a},y_{b}\right]  -T_{2}\left[  \overline
{x}_{\overline{a}},\overline{y}_{\overline{b}}\right]  \right\vert \nonumber\\
&  \leq\left(  a_{2}+\theta A_{2}\right)  \left\vert x_{a}-\overline
{x}_{\overline{a}}\right\vert +\max\left\{  \frac{1}{\theta},b_{2}+\theta
B_{2}\right\}  \left\vert y_{b}-\overline{y}_{\overline{b}}\right\vert
\nonumber\\
&  =m_{21}\left(  \theta\right)  \left\vert x_{a}-\overline{x}_{\overline{a}%
}\right\vert +m_{22}\left(  \theta\right)  \left\vert y_{b}-\overline
{y}_{\overline{b}}\right\vert . \label{T2_eval}%
\end{align}
Now, both inequalities (\ref{T1_eval}), (\ref{T2_eval}) can be put together
and be rewritten equivalently as%
\[
\left[
\begin{array}
[c]{c}%
\left\vert T_{1}\left[  x_{a},y_{b}\right]  -T_{1}\left[  \overline
{x}_{\overline{a}},\overline{y}_{\overline{b}}\right]  \right\vert \\
\left\vert T_{2}\left[  x_{a},y_{b}\right]  -T_{2}\left[  \overline
{x}_{\overline{a}},\overline{y}_{\overline{b}}\right]  \right\vert
\end{array}
\right]  \leq M_{\theta}\left[
\begin{array}
[c]{c}%
\left\vert x_{a}-\overline{x}_{\overline{a}}\right\vert \\
\left\vert y_{b}-\overline{y}_{\overline{b}}\right\vert
\end{array}
\right]
\]
or using the vector-valued norm%
\[
\left\Vert T\left(  u\right)  -T\left(  \overline{u}\right)  \right\Vert
_{\left(  C\left[  0,1\right]  \times%
%TCIMACRO{\U{211d} }%
%BeginExpansion
\mathbb{R}
%EndExpansion
\right)  ^{2}}\leq M_{\theta}\left\Vert u-\overline{u}\right\Vert _{\left(
C\left[  0,1\right]  \times%
%TCIMACRO{\U{211d} }%
%BeginExpansion
\mathbb{R}
%EndExpansion
\right)  ^{2}},
\]
where $M_{\theta}$ is given by (\ref{m}) and assumed to be convergent to
zero.$~$The result follows now from Perov's fixed point theorem.
\end{proof}

\section{Existence of at least one solution}
\setcounter{equation}{0} In the beginning of this section, we give an
application of Schauder's fixed point theorem. More precisely, we show that the
existence of solutions to the problem (\ref{1}) follows from Schauder's fixed
point theorem in case that $f_{1},f_{2}$ satisfy some relaxed growth condition
of the type:%

\begin{equation}
\left\{
\begin{array}
[c]{l}%
\left\vert f_{1}(t,x,y)\right\vert \leq a_{1}\left\vert x\right\vert
+b_{1}\left\vert y\right\vert +c_{1},\\
\left\vert f_{2}(t,x,y)\right\vert \leq a_{2}\left\vert x\right\vert
+b_{2}\left\vert y\right\vert +c_{2},
\end{array}
\right.  \label{300_2}%
\end{equation}
for all $x,y,\overline{x},\overline{y}\in%
%TCIMACRO{\U{211d} }%
%BeginExpansion
\mathbb{R}
%EndExpansion
,$ and%
\begin{equation}
\left\{
\begin{array}
[c]{l}%
\left\vert \alpha\left[  x,y\right]  \right\vert \leq A_{1}\left\vert
x\right\vert _{C}+B_{1}\left\vert y\right\vert _{C}+C_{1},\\
\left\vert \beta\left[  x,y\right]  \right\vert \leq A_{2}\left\vert
x\right\vert _{C}+B_{2}\left\vert y\right\vert _{C}+C_{2},
\end{array}
\right.  \label{301_2}%
\end{equation}
for all $x,y,\overline{x},\overline{y}\in$\ $C[0,1].$\medskip

\begin{theorem}
If the conditions \emph{(\ref{300_2}), (\ref{301_2})} hold and the matrix
\emph{(\ref{m}) }is convergent to zero for some $\theta>0,$ then the problem
\emph{(\ref{1})} has at least one solution.
\end{theorem}

\begin{proof}
In order to apply Schauder's fixed point theorem, we look for a nonempty,
bounded, closed and convex subset $B$ of $\left(  C\left[  0,1\right]  \times%
%TCIMACRO{\U{211d} }%
%BeginExpansion
\mathbb{R}
%EndExpansion
\right)  ^{2}$ so that $T(B)\subset B.$ Let $x_{a},y_{b}$ be any elements of
$C\left[  0,1\right]  \times%
%TCIMACRO{\U{211d} }%
%BeginExpansion
\mathbb{R}
%EndExpansion
.$\newline Then, using the same norm on $C\left[  0,1\right]  \times%
%TCIMACRO{\U{211d} }%
%BeginExpansion
\mathbb{R}
%EndExpansion
$ as in the proof of the previous theorem, we obtain
\begin{align}
&  \left\vert T_{1}\left[  x_{a},y_{b}\right]  \right\vert =\left\vert
a+\int_{0}^{t}f_{1}\left(  s,x\left(  s\right)  ,y(s)\right)  ds\right\vert
_{C}+\theta\left\vert \alpha\lbrack x,y]\right\vert \label{T1_eval2}\\
&  \leq\left\vert a\right\vert +\left\vert \int_{0}^{t}\left(  a_{1}\left\vert
x(s)\right\vert +b_{1}\left\vert y(s)\right\vert +c_{1}\right)  ds\right\vert
_{C}+\theta A_{1}\left\vert x\right\vert _{C}+\theta B_{1}\left\vert
y\right\vert _{C}+\theta C_{1}\nonumber\\
&  \leq a_{1}\left\vert x\right\vert _{C}+b_{1}\left\vert y\right\vert
_{C}+c_{1}+\theta A_{1}\left\vert x\right\vert _{C}+\theta B_{1}\left\vert
y\right\vert _{C}+\theta C_{1}+\left\vert a\right\vert \nonumber\\
&  =\left(  a_{1}+\theta A_{1}\right)  \left\vert x\right\vert _{C}+\left(
b_{1}+\theta B_{1}\right)  \left\vert y\right\vert _{C}+\frac{1}{\theta}%
\cdot\theta\left\vert a\right\vert +c_{1}+\theta C_{1}\nonumber\\
&  \leq\max\left\{  \frac{1}{\theta},a_{1}+\theta A_{1}\right\}  \left\vert
x_{a}\right\vert +\left(  b_{1}+\theta B_{1}\right)  \left\vert y_{b}%
\right\vert +c_{0}\nonumber\\
&  =m_{11}\left(  \theta\right)  \left\vert x_{a}\right\vert +m_{12}\left(
\theta\right)  \left\vert y_{b}\right\vert +c_{0},\nonumber
\end{align}
where $c_{0}:=c_{1}+\theta C_{1}.$ Similarly%
\begin{align}
\left\vert T_{2}\left[  x_{a},y_{b}\right]  \right\vert  &  \leq\left(
a_{2}+\theta A_{2}\right)  \left\vert x_{a}\right\vert +\max\left\{  \frac
{1}{\theta},b_{2}+\theta B_{2}\right\}  \left\vert y_{b}\right\vert
+C_{0}\label{T2_eval2}\\
&  =m_{21}\left(  \theta\right)  \left\vert x_{a}\right\vert +m_{22}\left(
\theta\right)  \left\vert y_{b}\right\vert +C_{0},\nonumber
\end{align}
where $C_{0}:=c_{2}+\theta C_{2}.$ Now, from (\ref{T1_eval2}), (\ref{T2_eval2}%
) we have%
\[
\left[
\begin{array}
[c]{c}%
\left\vert T_{1}\left[  x_{a},y_{b}\right]  \right\vert \\
\left\vert T_{2}\left[  x_{a},y_{b}\right]  \right\vert
\end{array}
\right]  \leq M_{\theta}\left[
\begin{array}
[c]{c}%
\left\vert x_{a}\right\vert \\
\left\vert y_{b}\right\vert
\end{array}
\right]  +\left[
\begin{array}
[c]{c}%
c_{0}\\
C_{0}%
\end{array}
\right]  ,
\]
where $M_{\theta}$ is given by (\ref{m}) and is assumed to be convergent to
zero. Next we look for two positive numbers $R_{1},R_{2}$ such that if
$\left\vert x_{a}\right\vert \leq R_{1}$ and $\left\vert y_{b}\right\vert \leq
R_{2},$ then $\left\vert T_{1}\left[  x_{a},y_{b}\right]  \right\vert \leq
R_{1},$ $\left\vert T_{2}\left[  x_{a},y_{b}\right]  \right\vert \leq R_{2}.$
To this end it is sufficient that
\[
M_{\theta}\left[
\begin{array}
[c]{c}%
R_{1}\\
R_{2}%
\end{array}
\right]  +\left[
\begin{array}
[c]{c}%
c_{0}\\
C_{0}%
\end{array}
\right]  \leq\left[
\begin{array}
[c]{c}%
R_{1}\\
R_{2}%
\end{array}
\right]  ,
\]
whence%
\[
\left[
\begin{array}
[c]{c}%
R_{1}\\
R_{2}%
\end{array}
\right]  \geq\left(  I-M_{\theta}\right)  ^{-1}\left[
\begin{array}
[c]{c}%
c_{0}\\
C_{0}%
\end{array}
\right]  .
\]
Notice that $I-M_{\theta}$ is invertible and its inverse $\left(  I-M_{\theta
}\right)  ^{-1}$ has nonnegative elements since $M_{\theta}$ is convergent to
zero. Thus, if $B=B_{1}\times B_{2},$ where
\[
B_{1}=\left\{  x_{a}\in C[0,1]\times%
%TCIMACRO{\U{211d} }%
%BeginExpansion
\mathbb{R}
%EndExpansion
:\left\vert x_{a}\right\vert \leq R_{1}\right\}  \text{ and }B_{2}=\left\{
y_{b}\in C[0,1]\times%
%TCIMACRO{\U{211d} }%
%BeginExpansion
\mathbb{R}
%EndExpansion
:\left\vert y_{b}\right\vert \leq R_{2}\right\}  ,
\]
then $T(B)\subset B$ and Schauder's fixed point theorem can be applied.
\end{proof}

In what follows, we give an application of the Leray-Schauder Principle and we
assume that the nonlinearlities $f_{1},f_{2}$ and also the functionals
$\alpha,\beta$ satisfy more general growth conditions, namely:%

\begin{equation}
\left\{
\begin{array}
[c]{l}%
\left\vert f_{1}(t,x,y)\right\vert \leq\omega_{1}(t,\left\vert x\right\vert
,\left\vert y\right\vert ),\\
\left\vert f_{2}(t,x,y)\right\vert \leq\omega_{2}(t,\left\vert x\right\vert
,\left\vert y\right\vert ),
\end{array}
\right.  \label{ls1}%
\end{equation}
for all $x,y\in%
%TCIMACRO{\U{211d} }%
%BeginExpansion
\mathbb{R}
%EndExpansion
,$
\begin{equation}
\left\{
\begin{array}
[c]{l}%
\left\vert \alpha\left[  x,y\right]  \right\vert \leq\omega_{3}(\left\vert
x\right\vert _{C},\left\vert y\right\vert _{C}),\\
\left\vert \beta\left[  x,y\right]  \right\vert \leq\omega_{4}(\left\vert
x\right\vert _{C},\left\vert y\right\vert _{C}),
\end{array}
\right.  \label{ls2}%
\end{equation}
for all $x,y\in C[0,1].$ Here $\omega_{1},\omega_{2}$ are $L^{1}%
$-Carath\'{e}odory functions on $[0,1]\times\mathbf{%
%TCIMACRO{\U{211d} }%
%BeginExpansion
\mathbb{R}
%EndExpansion
}_{+}^{2},~$nondecreasing in their second and third arguments, and $\omega
_{3},\omega_{4}$ are continuous functions on $\mathbf{%
%TCIMACRO{\U{211d} }%
%BeginExpansion
\mathbb{R}
%EndExpansion
}_{+}^{2},$ nondecreasing in both variables.

\begin{theorem}
Assume that the conditions \emph{(\ref{ls1}), (\ref{ls2})} hold. In addition
assume that there exists $R_{0}=\left(  R_{0}^{1},R_{0}^{2}\right)  \in\left(
0,\infty\right)  ^{2}$ such that for $\rho=\left(  \rho_{1},\rho_{2}\right)
\in\left(  0,\infty\right)  ^{2}$%
\begin{equation}
\left\{
\begin{array}
[c]{c}%
\int_{0}^{1}\omega_{1}(s,\rho_{1},\rho_{2})ds+\omega_{3}(\rho_{1},\rho
_{2})\geq\rho_{1}\\
\int_{0}^{1}\omega_{2}(s,\rho_{1},\rho_{2})ds+\omega_{4}(\rho_{1},\rho
_{2})\geq\rho_{2}%
\end{array}
\right.  \ \ \text{implies\ \ }\rho\leq R_{0}. \label{410_2}%
\end{equation}
Then the problem \emph{(\ref{1})} has at least one solution.
\end{theorem}

\begin{proof}
The result follows from the Leray-Schauder fixed point theorem once we have
proved the boundedness of the set of all solutions to equation $u=\lambda
T(u),$ for $\lambda\in\left(  0,1\right)  .$ Let $u=(x_{a},y_{b})$\ be such a
solution. Then $x_{a}=\lambda T_{1}(x_{a},y_{b})$ and $y_{b}=\lambda
T_{2}(x_{a},y_{b}),$ or equivalently%
\[
\left\{
\begin{array}
[c]{c}%
(x,a)=\lambda\left(  a+\int_{0}^{t}f_{1}\left(  s,x\left(  s\right)
,y(s)\right)  ds,\ \alpha\lbrack x,y]\right) \\
(y,b)=\lambda\left(  b+\int_{0}^{t}f_{2}\left(  s,x\left(  s\right)
,y(s)\right)  ds,\ \beta\lbrack x,y]\right)  .
\end{array}
\right.
\]
\newline First, we obtain that%
\begin{align}
\left\vert x(t)\right\vert  &  =\lambda\left\vert a+\int_{0}^{t}f_{1}\left(
s,x\left(  s\right)  ,y(s)\right)  ds\right\vert \nonumber\\
&  \leq\left\vert a\right\vert +\int_{0}^{t}\left\vert f_{1}\left(  s,x\left(
s\right)  ,y(s)\right)  \right\vert ds\nonumber\\
&  \leq\left\vert a\right\vert +\int_{0}^{1}\omega_{1}(s,\left\vert
(x(s)\right\vert ,\left\vert y(s)\right\vert )ds\nonumber\\
&  \leq\left\vert a\right\vert +\int_{0}^{1}\omega_{1}(s,\rho_{1},\rho_{2})ds
\label{omega1}%
\end{align}
where $\rho_{1}=\left\vert x\right\vert _{C},$ $\rho_{2}=\left\vert
y\right\vert _{C}.$ Also
\begin{equation}
\left\vert a\right\vert =\left\vert \lambda\alpha\lbrack x,y]\right\vert
\leq\omega_{3}(\rho_{1},\rho_{2}). \label{omega3}%
\end{equation}
Similarly, we have that%
\begin{equation}
\left\vert y(t)\right\vert \leq\left\vert b\right\vert +\int_{0}^{1}\omega
_{2}(s,\rho_{1},\rho_{2})ds \label{omega2}%
\end{equation}
and%
\begin{equation}
\left\vert b\right\vert \leq\omega_{4}(\rho_{1},\rho_{2}). \label{omega4}%
\end{equation}
Then from (\ref{omega1})-(\ref{omega4}), we deduce%
\[
\left\{
\begin{array}
[c]{l}%
\rho_{1}\leq\int_{0}^{1}\omega_{1}(s,\rho_{1},\rho_{2})ds+\omega_{3}(\rho
_{1},\rho_{2})\\
\rho_{2}\leq\int_{0}^{1}\omega_{2}(s,\rho_{1},\rho_{2})ds+\omega_{4}(\rho
_{1},\rho_{2}).
\end{array}
\right.
\]
This by (\ref{410_2}) guarantees that
\begin{equation}
\rho\leq R_{0}. \label{453_2}%
\end{equation}
It follows that
\begin{equation}
\left\vert a\right\vert \leq\omega_{3}(R_{0})=:R_{1}^{1},\text{ \ \ \ }%
\left\vert b\right\vert \leq\omega_{4}(R_{0})=:R_{1}^{2}. \label{rp}%
\end{equation}
Finally (\ref{453_2}) and (\ref{rp}) show that the solutions $u=\left(
x_{a},y_{b}\right)  $ are \textit{a priori} bounded independently on
$\lambda.$ Thus Leray-Schauder's fixed point theorem can be applied.
\end{proof}

\section{Numerical examples}

In what follows, we give two numerical examples that illustrate our theory.

\begin{example}
\label{exx1}{Consider the nonlocal problem}%
\begin{equation}
\left\{
\begin{array}
[c]{l}%
x^{\prime}=\frac{1}{4}\sin x+ay+g\left(  t\right)  \equiv f_{1}(t,x,y),\\
y^{\prime}=\cos\left(  ax+\frac{1}{4}y\right)  +h\left(  t\right)  \equiv
f_{2}(t,x,y),\\
x(0)=\frac{1}{8}\sin\left(  x\left(  \frac{1}{4}\right)  +y\left(  \frac{1}%
{4}\right)  \right)  ,\\
y(0)=\frac{1}{8}\cos\left(  x\left(  \frac{1}{4}\right)  +y\left(  \frac{1}%
{4}\right)  \right)  ,
\end{array}
\right.  \label{ex1}%
\end{equation}
where $t\in\left[  0,1\right]  ,$ $a\in\mathbb{R}$ and{\ }$g,h\in L^{1}\left(
0,1\right)  .${\ We have }$a_{1}=1/4,\ b_{1}=\left\vert a\right\vert
,\ a_{2}=\left\vert a\right\vert ,\ b_{2}=1/4$ and $A_{1}=B_{1}=A_{2}%
=B_{2}=1/8.${\ Consider }$\theta=2.$ {Hence }%
\begin{equation}
M_{\theta}=\left[
\begin{array}
[c]{cc}%
\frac{1}{2} & \left\vert a\right\vert +\frac{1}{4}\\
\left\vert a\right\vert +\frac{1}{4} & \frac{1}{2}%
\end{array}
\right]  .\label{star}%
\end{equation}
Since the eigenvalues of $M_{\theta}$\ are $\lambda_{1}=-\left\vert
a\right\vert +\frac{1}{4},~\lambda_{2}=\left\vert a\right\vert +\frac{3}{4},$
{the matrix (\ref{star}) is convergent to zero if }$\left\vert \lambda
_{1}\right\vert <1$ and $\left\vert \lambda_{2}\right\vert <1.$ It is also
known that a matrix of this type is convergent to zero if $\left\vert
a\right\vert +\frac{1}{4}+\frac{1}{2}<1$ (see \cite{p1}). Therefore, if
$\left\vert a\right\vert <\frac{1}{4},$ {the matrix (\ref{star}) is convergent
to zero and from Theorem 2.1 the problem (\ref{ex1}) has a unique solution.}
\end{example}

\begin{example}
\emph{Consider the nonlocal problem}%
\begin{equation}
\left\{
\begin{array}
[c]{l}%
x^{\prime}=\frac{1}{4}x\sin\left(  \frac{y}{x}\right)  +ay\sin\left(  \frac
{x}{y}\right)  +g(t)\equiv f_{1}(t,x,y),\\
y^{\prime}=ax\sin\left(  \frac{y}{x}\right)  +\frac{1}{4}y\sin\left(  \frac
{x}{y}\right)  +h(t)\equiv f_{2}(t,x,y),\\
x(0)=\frac{1}{8}\sin\left(  x\left(  \frac{1}{4}\right)  +y\left(  \frac{1}%
{4}\right)  \right)  ,\\
y(0)=\frac{1}{8}\cos\left(  x\left(  \frac{1}{4}\right)  +y\left(  \frac{1}%
{4}\right)  \right)  ,
\end{array}
\right.  \label{ex2}%
\end{equation}
where $t\in\left[  0,1\right]  ,$ $a\in\mathbb{R}${\ and }$g,h\in L^{1}\left(
0,1\right)  .${\ Since}%
\begin{align*}
\left\vert f_{1}\left(  t,x,y\right)  \right\vert  &  \leq\frac{1}%
{4}\left\vert x\right\vert +\left\vert a\right\vert \left\vert y\right\vert
+\left\vert g\left(  t\right)  \right\vert \\
\left\vert f_{2}\left(  t,x,y\right)  \right\vert  &  \leq\left\vert
a\right\vert \left\vert x\right\vert +\frac{1}{4}\left\vert y\right\vert
+\left\vert h\left(  t\right)  \right\vert
\end{align*}
{we are under the assumptions from the first part of Section 3. Also, the
matrix }$M_{\theta}${\ is that from Example \ref{exx1} if we consider }%
$\theta=2${. Therefore, according to Theorem 3.1, if that matrix is convergent
to zero}$,${ then the problem (\ref{ex2}) has at least one solution. Note that
the functions }$f_{1}\left(  t,x,y\right)  ,$ $f_{2}\left(  t,x,y\right)  $
{from this example do not satisfy Lipschitz conditions in }$x,y${\ and
consequently Theorem 2.1 does not apply.}
\end{example}

\section*{Acknowledgements}

This work was possible with the financial support of the Sectoral Operational
Programme for Human Resources Development 2007-2013, co-financed by the
European Social Fund, under the project number POSDRU/107/1.5/ S/76841 with
the title "Modern Doctoral Studies: Internationalization and
Interdisciplinarity", and by a grant of the Romanian National Authority for
Scientific Research, CNCS -- UEFISCDI, project number PN-II-ID-PCE-2011-3-0094.


\begin{thebibliography}{99}                                                                                               %


\bibitem{amp}E. Alves, T. F. Ma and M. L. Pelicer, Monotone positive
solutions for a fourth order equation with nonlinear boundary conditions,
\textit{Nonlinear Anal.} \textbf{71} (2009), 3834--3841.

\bibitem{ref12}R. P. Agarwal, M. Meehan and D. O'Regan, \textit{Fixed Point
Theory and Applications}, Cambridge University Press, Cambridge, 2001.

\bibitem{aizicovici-lee}S. Aizicovici and H. Lee, Nonlinear nonlocal Cauchy
problems in Banach spaces,\textit{ Appl. Math. Lett.} \textbf{18} (2005), 401--407.

\bibitem{avalshv1}G. Avalishvili, M. Avalishvili and D. Gordeziani, On
integral nonlocal boundary value problems for some partial differential
equations,\textit{ Bull. Georg. Acad. Sci.} \textbf{5} (1) (2011), 31-37.

\bibitem{avalshv2}G. Avalishvili, M. Avalishvili and D. Gordeziani, On a
nonlocal problem with integral boundary conditions for a multidimensional
elliptic equation,\textit{ Appl. Math. Lett.} \textbf{24} (2011), 566--571.

\bibitem{benchora-boucherif}M. Benchohra and A. Boucherif,\textit{\ }On first
order multivalued initial and periodic value problems, \textit{Dynam. Systems
Appl.} \textbf{9} (2000), 559-568.

\bibitem{bp}A. Berman and R.J. Plemmons, \textit{Nonnegative Matrices in the
Mathematical Sciences}, SIAM, Philadelphia, 1994.

\bibitem{bip-impulsive}O. Bolojan-Nica, G. Infante and P. Pietramala,
Existence results for impulsive systems with initial nonlocal
conditions,\textit{ Math. Model. Anal.} (submitted).

\bibitem{bip}O. Bolojan-Nica, G. Infante and R. Precup, Existence results for
systems with coupled nonlocal initial conditions, \textit{Nonlinear Anal.}, to appear.

\bibitem{ref4}A. Boucherif, Differential equations with nonlocal boundary
conditions, \textit{Nonlinear Anal.} \textbf{47} (2001), 2419-2430.

\bibitem{boucherif2}A. Boucherif, First-order differential inclusions with
nonlocal initial conditions,\textit{ Appl. Math. Lett.} \textbf{15} (2002), 409-414.

\bibitem{boucherif-ejde}A. Boucherif, Nonlocal Cauchy problems for
first-order multivalued differential equations,\textit{ Electron. J.
Differential Equations} \textbf{2012 }(2012), No. 47, 1-9

\bibitem{ref6}A. Boucherif and R. Precup, On the nonlocal initial value
problem for first order differential equations, \textit{Fixed Point Theory}
\textbf{4} (2003), 205-212.

\bibitem{bp2}A. Boucherif and R. Precup, Semilinear evolution equations with
nonlocal initial conditions, \textit{Dynam. Systems Appl.} \textbf{16} (2007), 507-516.

\bibitem{ref7}L. Byszewski, Theorems about the existence and uniqueness of
solutions of a semilinear evolution nonlocal Cauchy problem, \textit{J. Math.
Anal. Appl.} \textbf{162} (1991), 494-505.

\bibitem{ref8}L. Byszewski and V. Lakshmikantham, Theorem about the existence
and uniqueness of a solution of a nonlocal abstract Cauchy problem in a Banach
space, \textit{Appl. Anal.} \textbf{40} (1990), 11-19.

\bibitem{ref10-byz}L. Byszewski, Abstract nonlinear nonlocal problems and
their physical interpretation, in ``\textit{Biomathematics, Bioinformatics and
Applications of Functional Differential Difference Equations}'', H. Akca,
V. Covachev and E. Litsyn, Eds., Akdeniz Univ. Publ., Antalya, Turkey, 1999.

\bibitem{Cabada1}A. Cabada, An overview of the lower and upper solutions
method with nonlinear boundary value conditions, \textit{Bound. Value Probl.}
(2011), Art. ID 893753, 18 pp.

\bibitem{Cab-Ter}A. Cabada and S. Tersian, Multiplicity of solutions of a two
point boundary value problem for a fourth-order equation, \textit{Appl. Math.
Comput. }\textbf{219} (2013), 5261--5267.

\bibitem{acfm-nonlin}A. Cabada and F. Minh{\'{o}}s, Fully nonlinear
fourth-order equations with functional boundary conditions, \textit{J. Math.
Anal. Appl.} \textbf{340} (2008), 239--251.

\bibitem{deimling}K. Deimling, \textit{Multivalued Differential Equations},
W. de Gruyter, Berlin, 1992.

\bibitem{dfdorjp}D. Franco, D. O'Regan and J. Per\'{a}n, Fourth-order
problems with nonlinear boundary conditions, \textit{J. Comput. Appl. Math.}
\textbf{174} (2005), 315--327.

\bibitem{ref9}M. Frigon, \textit{Application de la th\'{e}orie de la
transversalite topologique a des problemes non linearies pour des equations
differentielles ordinaires}, Dissertationes Math. 296, PWN, Warsawa, 1990.

\bibitem{frig-lee}M. Frigon and J.W. Lee, Existence principle for
Carath\'{e}odory differential equations in Banach spaces, \textit{Topol.
Methods Nonlinear Anal.} \textbf{1 }(1993), 95-111.

\bibitem{Goodrich3}C. S. Goodrich, On nonlocal BVPs with nonlinear boundary
conditions with asymptotically sublinear or superlinear growth, \textit{Math.
Nachr.} \textbf{285} (2012), 1404--1421.

\bibitem{Goodrich4}C. S. Goodrich, Positive solutions to boundary value
problems with nonlinear boundary conditions, \textit{Nonlinear Anal.}
\textbf{75} (2012), 417--432.

\bibitem{Goodrich5}C. S. Goodrich, On nonlinear boundary conditions
satisfying certain asymptotic behavior, \textit{Nonlinear Anal.} \textbf{76}
(2013), 58--67.

\bibitem{han-park}H-K. Han and J-Y. Park, Boundary controllability of
differential equations with nonlocal conditions, \textit{J. Math. Anal. Appl.}
\textbf{230} (1999), 242-250.

\bibitem{gi-caa}G. Infante, Nonlocal boundary value problems with two
nonlinear boundary conditions, \textit{Commun. Appl. Anal.} \textbf{12}
(2008), 279--288.

\bibitem{infante-pietramala2}G. Infante, F.M. Minh\'{o}s and P. Pietramala,
Non-negative solutions of systems of ODEs with coupled boundary
conditions,\textit{ Commun. Nonlinear Sci. Numer. Simul. }\textbf{17} (2012), 4952-4960.

\bibitem{gipp-cant}G. Infante and P. Pietramala, A cantilever equation with
nonlinear boundary conditions \textit{Electron. J. Qual. Theory Differ. Equ.}
\textbf{15} (2009), 1--14.

\bibitem{infante-pietramala1}G. Infante and P. Pietramala, Eigenvalues and
non-negative solutions of a system with nonlocal BCs, \textit{Nonlinear Stud.}
\textbf{16 }(2009), No. 2, 187-196.

\bibitem{gipp-nonlin}G. Infante and P. Pietramala, Existence and multiplicity
of non-negative solutions for systems of perturbed Hammerstein integral
equations, \textit{Nonlinear Anal.}, \textbf{71} (2009), 1301--1310.

\bibitem{gipp-mmas}G. Infante and P. Pietramala, Multiple positive solutions
of systems with coupled nonlinear BCs, \textit{Math. Methods Appl. Sci.}, to appear.

\bibitem{jackson}D. Jackson, Existence and uniqueness of solutions to
semilinear nonlocal parabolic equations,\textit{ J. Math. Anal. Appl.}
\textbf{172} (1993), 256-265.

\bibitem{jankowski}T. Jankowski, Ordinary differential equations with
nonlinear boundary conditions,\textit{ Georgian Math. J.} \textbf{9} (2)
(2002), 287--294.

\bibitem{karakostas}G.L. Karakostas and P.Ch. Tsamatos, Existence of multiple
positive solutions for a nonlocal boundary value problem, \textit{Topol.
Methods Nonlinear Anal.} \textbf{19} (2002), 109-121.

\bibitem{prec-nica}O. Nica and R. Precup, On the nonlocal initial value
problem for first order differential systems, \textit{Stud. Univ.
Babe\c{s}-Bolyai Math.} \textbf{56 }(2011), No. 3, 125--137.

\bibitem{nica-EJDE}O. Nica, Initial-value problems for first-order
differential systems with general nonlocal conditions,\textit{ Electron. J.
Differential Equations} \textbf{2012 }(2012), No. 74, 1-15.

\bibitem{nica-DEA}O. Nica, Existence results for second order three-point
boundary value problems, \textit{Differ. Equ. Appl.} \textbf{4} (2012), 547-570.

\bibitem{nica-func-w}O. Nica, Nonlocal initial value problems for first order
differential systems, \textit{Fixed Point Theory }\textbf{13} (2012), 603--612.

\bibitem{nt}S.K. Ntouyas and P.Ch. Tsamatos, Global existence for semilinear
evolution equations with nonlocal conditions, \textit{J. Math. Anal. Appl.}
\textbf{210 }(1997), 679-687.

\bibitem{op}D. O'Regan and R. Precup, \textit{Theorems of Leray-Schauder Type
and Applications}, Gordon and Breach, Amsterdam, 2001.

\bibitem{paola}P. Pietramala, A note on a beam equation with nonlinear
boundary conditions, \textit{Bound. Value Probl.} (2011), Art. ID 376782, 14 pp.

\bibitem{ref11}R. Precup, \textit{Methods in Nonlinear Integral Equations},
Kluwer, Dordrecht, 2002.

\bibitem{p1}R. Precup, The role of matrices that are convergent to zero in
the study of semilinear operator systems, \textit{Math. Comp. Modelling}
\textbf{49} (2009), 703-708.

\bibitem{v}R.S. Varga, \textit{Matrix Iterative Analysis, Second Edition,}
Springer, Berlin, 2000.

\bibitem{webb-lan}J.R.L. Webb and K.Q. Lan, Eigenvalue criteria for existence
of multiple positive solutions of nonlinear boundary value problems of local
and nonlocal type, \textit{Topol. Methods Nonlinear Anal}. \textbf{27 }(2006), 91-115.

\bibitem{webb}J.R.L. Webb, A unified approach to nonlocal boundary value
problems, \textit{Dynam. Systems Appl.} Vol. 5. Proceedings of the 5th
International Conference, Morehouse College, Atlanta, GA, USA, May 30-June 2,
2007, 510-515.

\bibitem{webb-infante}J.R.L. Webb and G. Infante, Positive solutions of
nonlocal initial boundary value problems involving integral conditions,
\textit{NoDEA Nonlinear Differential Equations Appl.} \textbf{15} (2008), 45-67.

\bibitem{webb-infante1}J.R.L. Webb and G. Infante, Non-local boundary value
problems of arbitrary order,\textit{ J. London Math. Soc.} (2) \textbf{79}
(2009), 238--258.

\bibitem{webb-infante2}J.R.L. Webb and G. Infante, Semi-positone nonlocal
boundary value problems of arbitrary order, \textit{Commun. Pure Appl. Anal.}
(9) \textbf{2} (2010), 563-581.

\bibitem{xue1}X. Xue, Existence of semilinear differential equations with
nonlocal initial conditions,\textit{ Acta Math. Sin. (Engl. Ser.)} \textbf{23}
(6) (2007), 983--988.

\bibitem{xue2}X. Xue, Existence of solutions for semilinear nonlocal Cauchy
problems in Banach spaces, \textit{Electron. J. Differential Equations}
\textbf{2005 (}64\textbf{) }(2005), 1--7.
\end{thebibliography}
\end{document}